\documentclass[11pt]{amsart}

\usepackage{amsmath,amsthm,amsfonts,amssymb,bm,graphicx }

\usepackage{graphicx}
\usepackage{enumitem}
\usepackage{xparse}
\usepackage{tikz}
\usepackage{comment}
\usepackage{bm}
\usepackage[colorinlistoftodos]{todonotes}

\usepackage{cite}
\makeatletter
\newcommand{\citecomment}[2][]{\citen{#2}#1\citevar}
\newcommand{\citeone}[1]{\citecomment{#1}}
\newcommand{\citetwo}[2][]{\citecomment[,~#1]{#2}}
\newcommand{\citevar}{\@ifnextchar\bgroup{;~\citeone}{\@ifnextchar[{;~\citetwo}{]}}}
\newcommand{\citefirst}{\@ifnextchar\bgroup{\citeone}{\@ifnextchar[{\citetwo}{]}}}
\newcommand{\cites}{[\citefirst}
\makeatother

\usepackage{hyperref}
\hypersetup{colorlinks=true,citecolor=blue,linkcolor=blue,urlcolor=blue}

\usepackage{geometry}
\theoremstyle{plain}
\newtheorem{theorem}{Theorem}
\newtheorem*{theorem*}{Theorem}
\newtheorem{lemma}[theorem]{Lemma}

\newtheorem{cor}[theorem]{Corollary}
\newtheorem{proposition}[theorem]{Proposition}

\theoremstyle{remark}

\numberwithin{theorem}{section}

\numberwithin{equation}{section}

\def\Z{\mathbb Z}
\def\R{\mathbb R}
\def\Q{\mathbb Q}

\def\O{\mathcal O}

\def\co{\mathcal O}
\def\ve{\varepsilon}
\DeclareMathOperator{\Tr}{Tr}
\def\bt{\blacktriangle}

\begin{document}

\author{V\'it\v{e}zslav Kala}

\author{Magdal\'ena Tinkov\'a}

\title[Universal quadratic forms, small norms and traces]
{Universal quadratic forms, small norms and traces\\ in families of number fields}

\address{Charles University, Faculty of Mathematics and Physics, Department of Algebra,
Sokolovsk\'{a} 83, 18600 Praha 8, Czech Republic}
\email{kala@karlin.mff.cuni.cz, tinkova.magdalena@gmail.com}

\thanks{Both authors were supported by projects PRIMUS/20/SCI/002 and UNCE/SCI/022 from Charles University, and by grant 17-04703Y from Czech Science Foundation (GA\v{C}R). M. T. was further supported by project GA UK 1298218 (from Charles University) and by SVV-2017-260456.}



\begin{abstract}
  We obtain good estimates on the ranks of universal quadratic forms over Shanks' family of the simplest cubic fields and several other families of totally real number fields. As the main tool we characterize all the indecomposable integers in these fields and the elements of the codifferent of small trace. We also determine the asymptotics of the number of principal ideals of norm less than the square root of the discriminant.
\end{abstract}

\setcounter{tocdepth}{1}  
\maketitle 
\tableofcontents

\section{Introduction}

The study of representations of integers by quadratic forms is a subject with long, rich, and fascinating history that dates back at least to Babylonian clay tablet Plimpton 322 from around 1800 BC that lists 15 Pythagorean triples, i.e., representations of $0$ by the indefinite ternary form $x^2+y^2-z^2$.

The modern European history starts with giants such as Fermat, Euler, and Gauss, who considered representations of primes by binary definite forms $x^2+dy^2$ (for $d\in\Z_{>0}$), and Lagrange 
who in 1770 proved the Four Square Theorem stating that \emph{every positive integer $n$ is of the form $x^2+y^2+z^2+w^2$}. Thus he established that this quaternary form is \emph{universal} in the sense that it represents all positive integers.

After numerous other results, the theory culminated in the 15- and 290-Theorems of Conway, his students Miller, Schneeberger, and Simons, and Bhargava and Hanke \cite{Bh,BH} that state that \emph{a positive definite quadratic form is universal if and only if it represents $1,2,3,\dots, 290$} (or $1,2,3,\dots, 15$ provided that the form is \textit{classical}, i.e., all its non-diagonal coefficients are even).

\medskip

Besides from this success story, there has been considerable interest in universal quadratic forms over (rings of integers of) number fields, i.e., totally positive forms that represent all totally positive integers (for precise definitions, see Sections \ref{sec:prel} and \ref{sec:qforms}).

Maa\ss\@ \cite{Ma} in 1941 used theta series to prove that the sum of three squares is universal over $\Q(\sqrt 5)$.
Conversely, 
Siegel~\cite{Si3} in~1945 showed that the sum of any number of squares is universal only over the number fields $\Q, \Q(\sqrt 5)$. For our further discussion it will be interesting to note that \textit{indecomposable elements}, i.e., \textit{totally positive algebraic integers $\alpha$ that cannot be written as the sum $\alpha=\beta+\gamma$ of totally positive integers $\beta,\gamma$}, figured prominently in his proof (under the name ``extremal elements'').

In order to study universal forms over number fields, it is thus necessary to consider more general quadratic forms than just the sum of squares.

\medskip

Hsia, Kitaoka, and Kneser~\cite{HKK} in 1978 established a version of Local-Global Principle for universal forms over number fields. In particular, their Theorem 3 implies that a (totally positive) universal form exists over every totally real number field $K$.

This was then followed by numerous results on the structure of universal forms over $K$, and in particular, on the existence of universal forms of small rank $r$. It is easy to see that there is
never a universal form of rank $r=1$ or $2$. Moreover, when the degree $d$ of $K$ is odd, it quickly follows from Hilbert Reciprocity Law that there is no ternary universal form~\cite{EK}.

As we have seen with the universality of the sum of three squares over $\Q(\sqrt 5)$, ternary universal forms may exist in even degrees. Nevertheless, Kitaoka formulated the influential Conjecture that \emph{there are only finitely many totally real fields $K$ admitting a ternary universal form}.

Motivated by Kitaoka's conjecture, 
Chan, M.-H.~Kim, and Raghavan~\cite{CKR} found all classical universal forms over real quadratic number fields $\Q(\sqrt D)$ -- they exist only when $D=2,3,5$.
Several other authors investigated the universality of forms of other small ranks over specific real quadratic fields, in particular, Deutsch~\cite{De0,De,De2}, Lee~\cite{Le}, and Sasaki~\cite{Sa}. See also the nice survey by M.-H.~Kim~\cite{Km}.

Considering infinite families of real quadratic fields $K=\Q(\sqrt D)$,
B.~M.~Kim~\cite{Ki,Ki2} proved that there are only finitely many $K$ over which there is a diagonal 7-ary universal quadratic form, and constructed explicit 8-ary diagonal universal forms for each squarefree $D=n^2-1$.

Blomer and Kala \cite{BK,Ka} then found infinitely many real quadratic fields without universal forms of a given rank $r$; Kala and Svoboda \cite{KS} later extended these results also to multiquadratic fields.
Blomer and Kala \cite{BK2} also generalized Kim's construction to all real quadratic fields and obtained lower and upper bounds on the ranks of diagonal universal forms. 

\v Cech, Lachman, Svoboda, Tinkov\' a, and Zemkov\' a \cite{CLSTZ} investigated the structure of indecomposables in biquadratic fields. Their work was then followed by Kr\' asensk\' y,  Tinkov\' a, and Zemkov\' a \cite{KTZ} who proved Kitaoka's Conjecture for biquadratic fields.

\medskip

A key factor behind most of this recent progress on universal forms over (multi-)quadratic fields has been the fact that we explicitly understand the structure of their indecomposables: they can be nicely constructed from the periodic continued fraction for $\sqrt D$ \cite{DS}, see Section \ref{sec:quadr} below.
Analogous results in higher degree fields unfortunately are not available and, in full generality, seem very far out of reach at the moment; Brunotte's general upper bound on their norms \cite{Bru} is unfortunately too rough.

Yatsyna and Kala \cite{Ya,KY} managed to partly sidestep this issue by working with elements (of the codifferent) of trace one as they can be approached geometrically using interlacing polynomials and analytically via Siegel's Formula \cite{Si1} for 
the special value $\zeta_K(-1)$ of the Dedekind $\zeta$-function.

\medskip

In this article, we partly combine both approaches. We completely characterize indecomposables in certain families of cubic fields, and we determine their minimal traces after multiplication by elements of the codifferent.

Our method is fairly general, but for concreteness, we primarily focus on Shanks' family of the \textit{simplest cubic fields} \cite{Sh}, defined as $K=\Q(\rho)$, where $\rho$ is the largest root of the polynomial $$x^3-ax^2-(a+3)x-1\text{ for some fixed }a\in\Z_{\geq -1}.$$
These fields have a number of advantageous properties: they are Galois, possess units of all signatures, and the ring of integers $\co_K$ equals $\Z[\rho]$ for a positive density of $a$, namely, at least for those $a$ for which the square root of the discriminant $\Delta^{1/2}=a^2+3a+9$ is squarefree.
Their regulators are fairly small, and so the class numbers are quite large; 
overall the properties of Shanks' family were richly studied and also served as a prototype for the study of other families \cite{Ba,By,Cu,Fo,Kis,Lec,Let,Lo,Wa,Wa2}. Further relevant works dealt, e.g., with elliptic curves, diophantine equations, and numerous other topics \cite{Co,Du,Ho,Mi,MPL,LPV,LW,Th2,XCZ}.


\pagebreak

Our main result on universal forms is the following:

\begin{theorem}\label{thm:main univ}
	Let $K$ be the simplest cubic field with
	$a\in\Z_{\geq -1}$ such that $\co_K=\Z[\rho]$. 
	Then
	\begin{itemize}
		\item there is a \emph{diagonal} universal form of rank $3(a^2+3a+6)$ over $K$,
		\item every \emph{classical} universal form over $K$ has rank at least $\frac{a^2+3a+8}6$,
		\item every (non-classical) universal form over $K$ has rank at least $\frac{\sqrt{a^2+3a+8}}{3\sqrt 2}$ if $a\geq 21$.
	\end{itemize}
\end{theorem}

Note that this gives very sharp bounds on the minimal ranks of classical universal forms, as we have shown that they are $\asymp a^2\sim\Delta^{1/2}$. Even the corresponding constants are explicit: the lower bound is $\sim a^2/6$ and the upper bound is $\sim 3a^2$  (see Section \ref{sec:prel} for the analytic notation $\asymp,\ll,\sim$).
For non-classical universal forms we obtain a lower bound of magnitude $\asymp a\sim\Delta^{1/4}$.

In particular, we see that if $a\geq 3$, then there is no ternary classical universal form, and that the same holds for non-classical forms if $a\geq 21$, confirming Kitaoka's Conjecture for this family.

These results go far beyond the previous state of the art in two directions at once: They deal with a family of fields in degree greater than $2$ and moreover provide sharp bounds on the minimal ranks.

\medskip

The main technical tool for the proof of these results is the following characterization:

\begin{theorem} \label{thm:main}
Let $K$ be the simplest cubic field with the parameter
$a\in\Z_{\geq -1}$ such that $\co_K=\Z[\rho]$. 
The elements $1$, $1+\rho+\rho^2$, and $-v-w\rho+(v+1)\rho^2$ where $0\leq v\leq a$ and $v(a+2)+1\leq w\leq (v+1)(a+1)$ are, up to multiplication by totally positive units, all the indecomposable elements in $\co_K$.
\end{theorem}

For the proof, we establish that there are no other indecomposables by a geometry of numbers argument in Section \ref{sec:parall}. The indecomposability of these elements is then proved in Section \ref{sec:indeco} using the codifferent $\O_K^{\vee}$: For most of these elements $\alpha$ we find a totally positive element $\delta\in\O_K^{\vee,+}$ such that $\Tr(\alpha\delta)=1$, which immediately implies that $\alpha$ is indecomposable. For if $\alpha=\beta+\gamma$, then 
$1=\Tr(\alpha\delta)=\Tr(\beta\delta)+\Tr(\gamma\delta)\geq 1+1=2$, a contradiction.

However, we also show that in each (monogenic) simplest cubic field, the indecomposable $\alpha=1+\rho+\rho^2$ has minimal trace $2$:
$$\min_{\delta\in\O_K^{\vee,+}}\textup{Tr}(\alpha\delta)=2.$$
This is in sharp contrast to indecomposables in real quadratic fields that all have trace 1 by Proposition \ref{pr:quad}
and indicates that Yatsyna's approach \cite{Ya} based on interlacing polynomials probably cannot give all indecomposables in general.

In general, elements of the codifferent that have small trace are an important object of study; for example, they prominently figure in the aforementioned Siegel's Formula \cite{Si1} for the special value $\zeta_K(-1)$ of the Dedekind $\zeta$-function.

\medskip

Further, we show how one can use indecomposables to count elements of small norm.
Lemmermeyer and Peth\" o \cite[Theorem 1]{LP} proved that the smallest norm of a primitive non-unit element of a simplest cubic field is $2a+3$ -- we extend their result up to norm $a^2$.

For $X\in\R$ let $\mathcal P_a(X)$ be the number of primitive principal ideals $I$ (in a given monogenic simplest cubic field $K$) with norm $N(I)\leq X$. Recall that an ideal $I$ (or an element $\beta$) is primitive if $n\nmid I$ (or $n\nmid \beta$) for all $n\in\Z_{\geq 2}$.

\begin{theorem}\label{th count}	Let $K$ be the simplest cubic field with the parameter
	$a\in\Z_{\geq -1}$ such that $\co_K=\Z[\rho]$.  
	Let $\delta\in [0,1]$. Then $\mathcal P_a(a^{1+\delta})\asymp a^{2\delta/3}$.	
\end{theorem}

This Theorem gives quite a strong and surprising result! 
The discriminant of $K$ is $\Delta\sim a^4$, and so we are counting principal ideals of norms $<\Delta^{1/2}$.
Moreover, by the Class Number Formula, the class number $h$ is roughly of size $a^2\sim \Delta^{1/2}$ (up to logarithmic factors and the $L$-function value).

Thus, assuming GRH and the equidistribution of ideals in the class group, we should expect to have $\Delta^{1/2}/h\ll(\log b)^c$ primitive principal ideals of norm $<\Delta^{1/2}$ -- whereas Theorem \ref{th count} shows that there are $\asymp \Delta^{1/6}$ such ideals!

This is (again) in stark contrast with the situation in families of quadratic fields, where the number of primitive principal ideals of norm $<\Delta^{1/2}$ roughly obeys the prediction.
We prove both of these results in Section \ref{sec:small norms}. The proofs here (as well as in Sections \ref{sec:parall} and \ref{sec:indeco}) require quite a large amount of easy, yet somewhat tedious calculations, so we omit some of the details -- hopefully we have managed to strike the right balance between readability and completeness of the arguments.

\medskip

To conclude the Introduction, note that our methods are fairly robust and should apply to a broad range of families of number fields; here we primarily focused on the simplest cubic fields for concreteness and in order to avoid unnecessary technicalities. As an illustration, we briefly discuss results for two other cubic families in Section \ref{sec:other}.
Both behave quite differently from the simplest fields: they have only $\asymp \Delta^{1/4}$ indecomposables, and so the ranks of universal forms are also quite small. The indecomposables also typically (seem to) have trace $>1$;
in particular, in Subsection \ref{subsec:trace3} we find an element with minimal trace 3 that motivates us to formulate several open questions concerning this minimal trace.

As for the main simplifying properties of the simplest cubic fields, the use of \textit{monogenicity} 
can be easily sidestepped by working with non-maximal orders that have a power basis. But even for the full ring of integers $\co_K$, the main problem with the absence of a power basis seems to be the more complicated structure of the codifferent which should not be too critical. The fact that our fields are \textit{Galois} provided mostly a (very pleasant!) technical simplification. Similarly with the \textit{existence of units of all signatures}: this is important mostly in the proof of Theorem \ref{th count} that in general can be proved by working with indecomposable elements of all possible signatures.

In a sense we were lucky that the structure of indecomposables turned out to be as simple as Theorem \ref{thm:main}, rather than having a complicated shape as in a typical real quadratic field with long period length $s$. One could of course conversely use this to give a general definition of a ``simple field'' (or simple family) as one in which the set of indecomposables (up to multiplication by units) is the union of a ``small'' number $c$ of convex sets, where $c$ depends at most on the degree $d$ of the number field ($c=2$ in the simplest cubic case, and there are quadratic families with $c=1$). Characterizing all number fields with a given value of $c$ seems like a challenging, but very worthwhile, task.

When the number $c$ of convex components of indecomposables is large, one will most likely need to refine the geometric arguments of Section \ref{sec:parall} with better apriori understanding of the structure of indecomposables, perhaps by using a generalization of continued fractions.
For the simplest cubic fields we have succeeded in this task using periodic Jacobi-Perron algorithm -- we hope to extend this also to other families in a forthcoming article.

Hence we believe that extensions and generalizations of this article will turn out to be a fruitful field of work!

\section*{Acknowledgments}

We thank Valentin Blomer and Pavlo Yatsyna for our enriching discussions about the topics of this paper, in particular, about Sections \ref{sec:small norms} and \ref{sec:qforms}, respectively.

\section{Preliminaries}\label{sec:prel}

Let $K$ be a totally real number field of degree $d=[K:\Q]$ and $\O_K$ its ring of algebraic integers. The number field $K$ is \textit{monogenic} if there is a power basis for $\co_K$, i.e., $\co_K=Z[\eta]$ for some $\eta$.

The \emph{trace} of $\alpha\in K$ is $\text{Tr}(\alpha)=\sum_{\sigma}\sigma(\alpha)$ and the \emph{norm} of $\alpha$ is $N(\alpha)=\prod_{\sigma}\sigma(\alpha)$ where $\sigma$ in this sum and product runs over all the embeddings $\sigma:K\hookrightarrow\R$. 

We say that an element $\alpha\in K$ is \emph{totally positive} if $\sigma(\alpha)>0$ for all $\sigma$. We denote the set of totally positive algebraic integers of $K$ by $\O_K^{+}$. The unit group is $\co_K^\times$ and $\co_K^{\times,+}$ is the subgroup of totally positive units. Two elements $\alpha,\beta\in K$ are \textit{conjugate} if $\beta=\sigma(\alpha)$ for some $\sigma:K\hookrightarrow\R$, and \textit{associated} if $\beta=\varepsilon\alpha$ for some unit $\ve\in\co_K^\times$.

If $\co\subset \co_K$ is an \textit{order} in $K$ (i.e., a subring of finite index in $\co_K$), then we define the \textit{codifferent} of $\co$ as
\[
\O^{\vee}=\{\delta\in K\mid\text{Tr}(\alpha\delta)\in\Z\text{ for all }\alpha\in\O\}.
\]
If $\O=\Z[\eta]$ for some $\eta\in\co$, and $f$ is the minimal polynomial of $\eta$, then \cite[Proposition 4.17]{N}
\[
\O^{\vee}=\frac{1}{f'(\eta)}\co.
\]
In particular, this holds for the maximal order $\co_K$ if the field $K$ is {monogenic}.

An element $\alpha\in\O^{+}=\co\cap\co_K^+$ is called \emph{indecomposable in $\co$} if it cannot be expressed as $\alpha=\beta+\gamma$ for $\beta,\gamma\in \O^{+}$. Elements that are indecomposable in $\co_K$ are often just called indecomposables (or indecomposables in $K$).
For example, for $K=\Q$, we have only one indecomposable, namely $\alpha=1$.

$\O^{\vee,+}$ denotes the subset of totally positive elements of the codifferent. If for $\alpha\in\O^{+}$ we can find $\delta\in\O^{\vee,+}$ such that $\text{Tr}(\alpha\delta)=1$, then $\alpha$ is necessarily indecomposable.

\medskip

In this paper, we are mostly interested in indecomposable integers in the case of a simplest cubic field $K=\Q(\rho)$, where $\rho$ is the largest root of the polynomial 
$x^3-ax^2-(a+3)x-1$
for $a\geq -1$. In what follows, we use the notation $\rho, \rho'$, and $\rho''$ for the roots of this polynomial ordered so that $a+1<\rho$, $-2<\rho'<-1$, and $-1<\rho''<0$. Moreover, if $a\geq 7$, then \cite{LP} $$a+1<\rho<a+1+\frac{2}{a},\ \ -1-\frac{1}{a}<\rho'<-1-\frac{1}{2a},\text{ and } -\frac{1}{a+2}<\rho''<-\frac{1}{a+3}.$$  

We denote the conjugates of $\alpha\in K$ as $\alpha'$ and $\alpha''$;
these correspond to the embeddings given by $\rho\mapsto\rho'$ and $\rho\mapsto\rho''$, respectively. Moreover, the simplest cubic fields are totally real and Galois, and $\O_K=\Z[\rho]$ for infinitely many cases of $a$. Specifically,  $\O_K=\Z[\rho]$ if the square root $a^2+3a+9$ of the discriminant $\Delta=(a^2+3a+9)^2$ is squarefree; this condition is not necessary since for example for $a=0$, we have the discriminant equal to $9^2$ but at the same time $\O_K=\Z[\rho]$. In this paper, we deal primarily with this monogenic case.

Furthermore, we will use the fact that the group of units of $\Q(\rho)$ is generated by the pair $\rho$ and $\rho'$ \cite{God, Sh}. The simplest cubic fields contain units of all signs, and so every totally positive unit is a square in $\O_K$ \cite[p. 111, Corollary 3]{N}.

We use the common analytic number theory notation:
For two functions $f(x), g(x)$ we write $f\ll g$ (or $g\gg f$) if $|f(x)|<Cg(x)$ for some constant $C$ and all sufficiently large $x$, and $f\asymp g$ if $f\ll g$ and $f\gg g$.
Finally, $f\sim g$ means that $\lim_{x\rightarrow\infty} f(x)/g(x)=1$.

\section{Elements of trace one in quadratic fields} \label{sec:quadr}

In this Section, we will focus on the case of quadratic fields for comparison.
Let us start by introducing some additional notation and preliminaries.

For a real quadratic field $K=\Q(\sqrt{D})$ where $D>1$ is a squarefree positive integer, we have full characterization of indecomposable integers. To state it, let
\[
\omega_D=\left\{
\begin{array}{l}
\sqrt{D} \\
\frac{1+\sqrt{D}}{2}
\end{array}
\right.\text{ and }
\xi_D=-\omega'_D=
\left\{
\begin{array}{ll}
\sqrt{D} \qquad\text{if } D\equiv 2,3\;(\text{mod } 4),\\
\frac{\sqrt{D}-1}{2}\,\quad\text{if } D\equiv 1\;(\text{mod } 4).
\end{array}
\right.
\]
Then $\O_K=Z[\omega_D]$. Let $\Delta=\Delta_K$ be the discriminant of $K$, i.e., $\Delta=D$ if $D\equiv 1\pmod 4$ and $\Delta=4D$ otherwise.
For $\alpha\in K$ we denote its conjugate $\alpha'$.

\medskip

Let $\xi_D=[u_0,\overline{u_1,u_2,\ldots,u_s}]$ be the continued fraction of $\xi_D$ and define
\[
\frac{p_i}{q_i}=[u_0,\ldots,u_i]\text{ for coprime positive integers }p_i,q_i,i\geq 0; p_{-1}=1,q_{-1}=0.
\]
Then the algebraic integers of the form $\alpha_i=p_i+q_i\omega_D$ with $i\geq -1$ are called the \emph{convergents} (by a slight, but convenient, abuse of the usual terminology where convergents are the fractions $p_i/q_i$). We will also consider the \emph{semiconvergents}, i.e., the elements of the form $\alpha_{i,r}=\alpha_{i}+r\alpha_{i+1}$ with $i\geq -1, 0\leq r< u_{i+2}$. 

All indecomposables in $\Q(\sqrt{D})$ are exactly these elements $\alpha_{i,r},\alpha'_{i,r}$ with $i$ odd, i.e., the totally positive semiconvergents and their conjugates \cites[Theorem 2]{DS}[§16]{Pe}. 
Note that $\alpha_{i,u_{i+2}}=\alpha_{i+2,0}$ is also an indecomposable, so one can allow $r= u_{i+2}$.

Since the quadratic field $K$ is monogenic, the codifferent is
\[
    \O_K^{\vee}=\frac1{\sqrt \Delta}\co_K=
    \left\{
                \begin{array}{ll}
                  \frac{1}{2\sqrt{D}}\Z[\sqrt{D}] \;\,\quad\text{\ \ if } D\equiv 2,3\;(\text{mod } 4),\\
                  \frac{1}{\sqrt{D}}\Z\left[ \frac{1+\sqrt{D}}{2}\right] \quad\text{if } D\equiv 1\;(\text{mod } 4).
                \end{array}
              \right.
\]

Let us now characterize the quadratic indecomposables as trace one elements; we will later apply this result to universal quadratic forms in Subsection \ref{subsec:7.3}.

\begin{proposition}\label{pr:quad}
Let $K=\Q(\sqrt{D})$. Then $\alpha\in\O_K^+$ is indecomposable if and only if $$\min_{\delta\in\O_K^{\vee,+}}\textup{Tr}(\alpha\delta)=1.$$

Further, if $\alpha=\alpha_{i,r}$ with $0\leq r\leq u_{i+2}$, then the choice of $\delta$ depends only on $i$ (and not on $r$).
\end{proposition}

\begin{proof}
We need to prove that for every indecomposable integer, there exists $\delta\in\O_K^{\vee,+}$ such that  
$\textup{Tr}(\alpha\delta)=1$. First of all, suppose that $D\equiv 2,3\;(\text{mod }4)$. Then any element of $\O_K^{\vee}$ can be written as
\[
\delta=\frac{1}{2\sqrt{D}}(c+d\sqrt{D})=\frac{1}{2D}(dD+c\sqrt{D})
\]
for some $c,d\in\Z$. Let $\alpha$ be an indecomposable integer of the form $\alpha_{i,r}=p_i+rp_{i+1}+(q_i+rq_{i+1})\sqrt{D}$ for some odd $i$. It can be easily computed that
\[
\text{Tr}(\alpha\delta)=(p_{i}+rp_{i+1})d+c(q_{i}+rq_{i+1})
\]
If we put $d=q_{i+1}$ and $c=-p_{i+1}$, we get
\[
\text{Tr}(\alpha\delta)=-p_{i+1}q_i+p_iq_{i+1}+r(p_{i+1}q_{i+1}-p_{i+1}q_{i+1})=(-1)^{1+i}=1,
\] 
where we used the well-known identity $p_iq_{i-1}-p_{i-1}q_i=(-1)^{i-1}$.

It remains to show that $\delta$ is totally positive. We can rewrite $\delta$ as
\[
\delta=\frac{\sqrt{D}}{2D}(-p_{i+1}+q_{i+1}\sqrt{D}).
\]
Since $p_{i+1}-q_{i+1}\sqrt{D}<0$ for $i$ odd, we get $\delta>0$. Consequently,
\[
\delta'=\frac{-\sqrt{D}}{2D}(-p_{i+1}-q_{i+1}\sqrt{D})=\frac{\sqrt{D}}{2D}(p_{i+1}+q_{i+1}\sqrt{D})>0,
\]
and so $\delta$ is indeed totally positive. 

The case $D\equiv 1\;(\text{mod }4)$ is analogous; for indecomposable integer $\alpha_{i,r}$, we consider the element of the codifferent of the form
\[
\delta=-\sqrt{D}\left(p_{i+1}+q_{i+1}\frac{1-\sqrt{D}}{2}\right),
\]
which is totally positive and gives the desired trace $1$ of the product with $\alpha_{i,r}$.
\end{proof}

\section{Parallelepipeds generated by systems of totally positive units}\label{sec:parall}

Let us now outline a general method for characterizing indecomposables in all fields in a suitable family that we will then apply to the case of the simplest cubic fields.

Consider the diagonal embedding  $\sigma:K\rightarrow \R^d, \alpha\mapsto (\sigma_1(\alpha),\dots,\sigma_d(\alpha))$ into the Minkowski space (where $\sigma_1,\dots,\sigma_d$ are all the real embeddings of $K$), and the totally positive octant $\R^{d,+}=\{(v_1,\dots,v_d)\mid v_i>0\text{ for all }i\}$.

Let us further consider a fundamental domain for the action of 
multiplication by (the images of) totally positive units $\ve\in\co_K^{\times,+}$ on $\R^{d,+}$. 
By Shintani's Unit Theorem \cite[Thm (9.3)]{Ne}, we can choose a polyhedric cone $\mathcal P$ as this fundamental domain.

Recall that a \emph{polyhedric cone} is a finite disjoint union of \emph{simplicial cones}, i.e., subsets of $\R^d$ of the form $\R^+\ell_1+\dots+\R^+\ell_e$, where $\ell_1,\dots,\ell_e$ are linearly independent vectors in $\R^d$. Note that as defined, we are considering open simplicial cones -- e.g., if $e>1$, then $\ell_i\not\in\R^+\ell_1+\dots+\R^+\ell_e$.

In fact, if we choose a system $\ve_1,\dots,\ve_{d-1}$ of totally positive fundamental units (i.e., of generators of the group $\co_K^{\times,+}$), then $\mathcal P$ is the union of simplicial cones that are generated by some of the vectors of the form $\sigma(\prod_{i\in I}\ve_i)$ for subsets $I\subset\{1,\dots, d\}$.

Thus every indecomposable integer $\alpha\in\co_K^+$ can be multiplied by some totally positive unit so that it lies in one of these simplicial cones.

\medskip

To count indecomposables, it will be more natural to work in the embedding obtained by fixing an integral basis $\omega_1,\dots,\omega_d$ for $K$ and defining
\begin{align*}
\tau:K\rightarrow &\ \R^d\\
\sum x_i\omega_i\mapsto &\ (x_1,\dots,x_d).
\end{align*}
Elements of $\co_K$ then map to lattice points in $\Z^d$.

The change between embeddings $\sigma$ and $\tau$ is given by an invertible linear transformation; and so we can naturally consider the polyhedric cone $\mathcal Q$ in the embedding $\tau$ that corresponds to $\mathcal P$ and its decomposition into a disjoint union of simplicial cones.

For short, we will use the following notations for the cone $\mathcal C(\alpha_1,\dots,\alpha_e)=\R^+\tau(\alpha_1)+\dots+\R^+\tau(\alpha_e)$ and for its topological closure $\overline{\mathcal C}(\alpha_1,\dots,\alpha_e)=\R_0^+\tau(\alpha_1)+\dots+\R_0^+\tau(\alpha_e)$.

\medskip

More concretely, let us consider the case of cubic fields, i.e., $d=3$.
For a suitable pair of units $\ve_1,\ve_2$, by \cite[Theorem 1]{ThV} (or \cite[Theorem 2]{Ok}) we can take, e.g., 
\begin{align*}
\mathcal Q&= \mathcal C(1,\ve_1,\ve_2)\sqcup \mathcal C(1,\ve_1,\ve_1\ve_2^{-1})\sqcup\mathcal C(1,\ve_1)\sqcup\mathcal C(1,\ve_2)\sqcup\mathcal C(1,\ve_1\ve_2^{-1})\sqcup\mathcal C(1)\\
&\subset
\overline{\mathcal C}(1,\ve_1,\ve_2)\cup \overline{\mathcal C}(1,\ve_1,\ve_1\ve_2^{-1}),
\end{align*}
where $\sqcup$ denotes disjoint union.
Note that $\mathcal C(1,\ve_1,\ve_2),\mathcal C(1,\ve_1,\ve_1\ve_2^{-1})$ are the two cones of maximal dimension.

Let us now consider an indecomposable $\alpha$ such that $\tau(\alpha)\in\mathcal Q$. It lies in one of the two closed cones, without loss of generality let
$\tau(\alpha)\in\overline{\mathcal C}(1,\ve_1,\ve_2)$. 
If all the coordinates of $\tau(\alpha)$ are greater than the corresponding coordinates of $\tau(\ve_1)$, then $\alpha\succ\ve_1$ and $\alpha$ is not indecomposable. And similarly for the other vertices $\tau(1),\tau(\ve_1)$.
Thus $\tau(\alpha)$ has to lie in the cut-off part of the cone $\overline{\mathcal C}(1,\ve_1,\ve_2)$. 
This part of the cone is in turn contained in the parallelepiped ${\mathcal D}(1,\ve_1,\ve_2)=[0,1]\tau(1)+[0,1]\tau(\ve_1)+[0,1]\tau(\ve_2)$.

To sum up this part, we have seen that, up to multiplication by totally positive units, the image each indecomposable $\alpha$ lies in one of the two parallelepipeds ${\mathcal D}(1,\ve_1,\ve_2)$, ${\mathcal D}(1,\ve_1,\ve_1\ve_2^{-1})$. As $\tau$ maps $\co_K$ to $\Z^d$, this image $\tau(\alpha)$ is a lattice point in the parallelepiped.

\medskip

Thus to count (or characterize) indecomposables, we just need to count lattice points in parallelepipeds. Although Pick's Formula relating the volume of a convex body to the number of lattice points inside does not hold in dimension $>2$ in general, it luckily does hold for parallelepipeds. The validity of the formula in this case was one the reasons (besides the fact that it is simply easier to work with parallelepipeds) why we enlarged the cut-off of the cone to the whole parallelepiped.

\begin{proposition}
	Let $\ell_1,\dots,\ell_e\in\Z^e$ be linearly independent vectors and let 
	${\mathcal P}(\ell_1,\dots,\ell_e)=[0,1)\ell_1+\dots+[0,1)\ell_e$ be the corresponding (semi-open) parallelepiped. Then the number of lattice points in it equals its volume, i.e., 
	$$\#\left({\mathcal P}(\ell_1,\dots,\ell_e)\cap\Z^d\right)=\mathrm{vol }({\mathcal P}(\ell_1,\dots,\ell_e))
	=|\det \left((\ell_{ij})_{1\leq i,j\leq e}\right)|,$$
	where $\ell_i=(\ell_{i1},\dots, \ell_{ie})$.
\end{proposition}

As we do not use this folklore Proposition in the paper, let us not include its proof and instead refer the reader to \cite{ah}.

\medskip

The general strategy for characterizing all indecomposables should now be clear. We will 
\begin{enumerate}
	\item suitably choose the two parallelepipeds, 
	\item find all the lattice points contained in them with help of the volume formula for their number (although this is not strictly necessary, one can just directly describe all the lattice points), and then
	\item see which of these lattice points in fact correspond to indecomposables.
\end{enumerate}

To carry out this strategy to find candidates for indecomposables in the simplest cubic fields, we will take 
$\ve_1=\rho^2$, $\ve_2=(\rho'')^{-2}$ (note that this pair of units, as well as the pair $\ve_1,\ve_1\ve_2^{-1}$ is \textit{proper} in the sense of \cite[Corollary 2]{ThV}). This corresponds to choosing the following two neighboring parallelepipeds: The first one is generated by units $1$, $\rho^2$, and $1+2\rho+\rho^2=(\rho'')^{-2}$, and the second one by $1$, $\rho^2$, and $-1-a-(a^2+3a+3)\rho+(a+2)\rho^2=(\rho')^{-2}$. All of these elements are obviously totally positive units in the simplest cubic field $\Q(\rho)$, and both parallelepipeds lie in the totally positive octant.

Let us now deal with them separately.

\subsection{Parallelepiped generated by $1$, $\rho^2$, and $1+2\rho+\rho^2$} \label{subsec:par1}   
These units have the coordinates $(1,0,0)$, $(0,0,1)$, and $(1,2,1)$ in the basis $1,\rho,\rho^2$. Now we aim to determine all the coordinates $(m,n,o)\in\Z^3$, for which we can find $t_1,t_2,t_3\in[0,1]$ such that
\[
t_1(1,0,0)+t_2(0,0,1)+t_3(1,2,1)=(m,n,o).
\] 
Considering the 2nd coordinate, we see that $t_3$ must be equal to $0,\frac{1}{2}$, or $1$. For $t_3=\frac{1}{2}$, we get the algebraic integer $1+\rho+\rho^2$. For all the other choices, we obtain either the considered units, or some sum of them, i.e., a decomposable integer. Thus the only non-unit candidate for an indecomposable  from this parallelepiped is $1+\rho+\rho^2$. 

\subsection{Parallelepiped generated by $1$, $\rho^2$, and $-1-a-(a^2+3a+3)\rho+(a+2)\rho^2$} \label{subsec:par2}

Next we will focus on the neighboring parallelepiped. For $a=-1$, we do not get any other elements except for units and their sums, and so assume $a\geq 0$. In that case, we have the equation
\[
t_1(1,0,0)+t_2(0,0,1)+t_3(-1-a,-a^2-3a-3,a+2)=(m,n,o)\in\Z^3.
\]
Obviously, $t_3$ is of the form $t_3=\frac{w}{a^2+3a+3}$ for some $w\in\Z$, $0\leq w\leq a^2+3a+3$; i.e., $n=-w$. For $w=0$ and $w=a^2+3a+3$, we obtain only the original units and their sums, thus we can assume $1\leq w \leq a^2+3a+2=(a+1)(a+2)$. 
We will further divide this interval into the disjoint union of subintervals 
$v(a+2)+1\leq  w \leq (v+1)(a+2)$ indexed by $0\leq v \leq a$.

Next we have
\[
t_1=m+w\frac{a+1}{a^2+3a+3}=m+w\frac{1}{a+2+\frac{1}{a+1}}\in [0,1],
\]
and so for $v(a+2)+1\leq  w \leq (v+1)(a+2)$ with $0\leq v \leq a$, we necessarily have $m=-v$. Similarly, for $t_2$ we have
\[
t_2=o-w\frac{a+2}{a^2+3a+3}=o-w\frac{1}{a+1+\frac{1}{a+2}}\in [0,1].
\]
Hence for fixed $v$, we have $o=v+1$ if $v(a+2)+1\leq w\leq (v+1)(a+1)$, and $o=v+2$ if $(v+1)(a+1)+1\leq w\leq (v+1)(a+2)$. In the second case, these elements are decomposable:

\begin{lemma}
The elements of the form $-v-w\rho+(v+2)\rho^2$ where $0\leq v \leq a$ and $(v+1)(a+1)+1\leq w\leq (v+1)(a+2)$ are decomposable.
\end{lemma} 

\begin{proof}
Let us decompose the considered elements as
\[
-v-w\rho+(v+2)\rho^2=[-v-(a+1)(v+1)\rho+(v+1)\rho^2]+[-(w-(a+1)(v+1))\rho+\rho^2].
\] 
Both of these summands are some of the elements found in the parallelepiped in this Subsection: In the case of the first one, it is obvious. Considering the second one, we can deduce that
\[
1\leq w-(a+1)(v+1)\leq (v+1)(a+2)-(v+1)(a+1)=v+1\leq a+1,
\]
i.e., we obtain it for $v'=0$. Thus these two summands are totally positive and we found a decomposition of $-v-w\rho+(v+2)\rho^2$ in $\co_K^+$.
\end{proof}

The candidates on indecomposables from this parallelepiped are therefore the elements $-v-w\rho+(v+1)\rho$ where $0\leq v\leq a$ and $v(a+2)+1\leq w\leq (v+1)(a+1)$. 
Let us now prove their indecomposability!
 
\section{Indecomposables in the simplest cubic fields}\label{sec:indeco}

To prove the indecomposability of the elements from Subsections \ref{subsec:par1} and \ref{subsec:par2}, we will use the codifferent. If $\O_K=\Z[\rho]$, we have
\[
\O_K^{\vee}=\frac{1}{3\rho^2-2a\rho-(a+3)}\Z[\rho]=\frac{1}{a^2+3a+9}\left( -4-a+(-1-2a)\rho+2\rho^2\right) \Z[\rho].
\]

\subsection{The triangle of indecomposables}\label{subsec:triangle}
Let us start with the elements from Subsection \ref{subsec:par2}. 

\begin{lemma} \label{lemma:indetria}
If $\O_K=\Z[\rho]$, then the elements of the form $\alpha=-v-w\rho+(v+1)\rho^2$ where $0\leq v\leq a$ and $v(a+2)+1\leq w\leq (v+1)(a+1)$ are indecomposable.

There is an element $\delta\in\co_K^{\vee,+}$ (independent of $\alpha$) such that $\Tr(\delta\alpha)=1$ for each such element $\alpha$.
\end{lemma}

\begin{proof}
By Section \ref{sec:parall}, all of these elements are totally positive (as they lie in the totally positive octant).

Let us consider the element of the codifferent of the form
\[
\delta=\frac{1}{a^2+3a+9}(-4-a+(-1-2a)\rho+2\rho^2)(-(a+2)-a\rho+\rho^2).
\]
We can easily check that $(-4-a+(-1-2a)\rho+2\rho^2)(-(a+2)-a\rho+\rho^2)$ is a root of the polynomial
$
x^3-(a^2+3a+9)x^2+2(a^2+3a+9)x-(a^2+3a+9).
$
All three roots of this polynomial are positive (e.g., by Descartes sign rule), and so
$\delta\in\O_K^{\vee,+}$.

For $\alpha=v_1+v_2\rho+v_3\rho^2\in\Z[\rho]$, a direct computation shows that
$
\text{Tr}(\alpha\delta)=v_1+v_3.
$
Thus for the totally positive elements $\alpha=-v-w\rho+(v+1)\rho^2$ we have $\Tr(\alpha\delta)=1$, and so they are indecomposable (as explained in the Introduction).
\end{proof}

Denote $\blacktriangle=\blacktriangle(a)$ the ``triangle'' of indecomposables from Lemma \ref{lemma:indetria}, i.e., 
\[
\bt(a)=\{-v-w\rho+(v+1)\rho^2\mid 0\leq v\leq a \text{ and } v(a+2)+1\leq w\leq (v+1)(a+1)\}.
\]

Observe that conjugation corresponds to the rotation of this triangle, more precisely:
Let $$\alpha=-v-w\rho+(v+1)\rho^2=-v-(v(a+2)+1+W)\rho+(v+1)\rho^2=\alpha(v,W)$$ for $W=w-v(a+2)-1$. Then multiplication of the conjugate $\alpha'$ by the totally positive unit $-1-a-(3+3a+a^2)\rho+(2+a)\rho^2$ corresponds to the transformation
\[
T_1(\alpha(v,W))=(\alpha(v,W))'(-1-a-(3+3a+a^2)\rho+(2+a)\rho^2)=\alpha(W,a-v-W).
\] 
Similarly, for $\alpha''$ and the unit $\rho^2$, we get
\[
T_2(\alpha(v,W))=(\alpha(v,W))''\rho^2=\alpha(a-v-W,v).
\]
It is easy to directly check from the inequalities defining the triangle $\bt$ that $T_1(\alpha), T_2(\alpha)\in\bt$ and that they correspond to the rotations around the center. E.g., $T_1$ maps the vertices of $\bt$ as
\[
T_1:\alpha(0,0)\mapsto\alpha(0,a)\mapsto\alpha(a,0)\mapsto\alpha(0,0).
\]
The element $\alpha(v,W)$ is fixed by $T_1$ (or equivalently, $T_2$) if and only if $v=W=\frac{a}{3}$ which is of course possible only if $3\mid a$.

Thus, besides $\alpha$, the triangle $\bt$ contains unit multiples of its conjugates $\alpha'$ and $\alpha''$, 
and these conjugates are not associated with $\alpha$ except when $\alpha$ is the center of the triangle $\alpha=\alpha(a/3,a/3)$. 

Note that we could also include the unit $\rho^2$ in the triangle, as it is obtained by choosing $v=0, w=0$ (that is excluded from Lemma \ref{lemma:indetria}), and similarly with its images $\alpha(-1,a+1),\alpha(a+1,0)$ under $T_1,T_2$.
In particular, these three units all have trace 1 after multiplication by the element $\delta$ from the proof of Lemma \ref{lemma:indetria}.

In Section \ref{sec:small norms} we will need to estimate the norms of indecomposables. First note that for $\alpha(v,W)\in\bt$ we have $0\leq v\leq a$, $0\leq W\leq a-v$, and
\begin{align*}
N(\alpha(v,W)) & =  a^2 v W - a v^2 W - a v W^2 + a^2 v - 2 a v^2 + a v W + a W^2 + v^3 - 
3 v W^2 - W^3\\ 
& + 3 a v + 3 a W - 3 v^2 - 3 v W - 3 W^2 + 2 a + 3.
\end{align*}

However, we will primarily be interested in indecomposables up to conjugation and association, and so we will need to consider only a third of the triangle $\bt$. To formulate this precisely, let  
$a=3A+a_0$ where $a_0\in\{0,1,2\}$. Put
\[
\bt_0(a)=
\begin{cases}
\left\{\alpha(v,W)\mid 0\leq v\leq A\quad \ \ \text{ and } v\leq W\leq 3A+a_0-2v-1\right \}&\text{ if } a_0\in\{1,2\},\\
\{\alpha(v,W)\mid 0\leq v\leq A-1\text{ and } v\leq W\leq 3A+a_0-2v-1\}\cup\{\alpha(A,A)\}&\text{ if } a_0=0.
\end{cases}
\]
As we saw above, it is easy to directly check that for each $\alpha\in\bt$ there is unique $\beta\in\bt_0$ such that $\alpha$ differs from a conjugate of $\beta$ by the multiple of a unit.

\subsection{The exceptional indecomposable} \label{subsec:except}
We now turn to the element $1+\rho+\rho^2$ from Subsection \ref{subsec:par1} and prove its indecomposability.

\begin{lemma} \label{lemma:indespec}
Assume that $\O_K=\Z[\rho]$. Then the element $1+\rho+\rho^2$ is indecomposable. Moreover, $\min_{\delta\in\O_K^{\vee,+}}\textup{Tr}((1+\rho+\rho^2)\delta)=2$. 
\end{lemma}

\begin{proof}
Recall that in Section \ref{sec:parall} we saw that $1+\rho+\rho^2$ is totally positive.

If we multiply $1+\rho+\rho^2$ by $\delta$, the element of codifferent from the proof of Lemma \ref{lemma:indetria}, we get an element of trace $2$. Thus, to prove that $1+\rho+\rho^2$ is indecomposable, we have to show that this element cannot be written as the sum of two totally positive elements that have trace 1 after multiplication by this $\delta$. We saw in the proof of Lemma \ref{lemma:indetria} that such elements are of the form $-v-w\rho+(v+1)\rho^2$. At least one of these summands must have $v\leq -1$, 
but then one directly shows that such an element is totally positive only if $v=-1,w=0$, when the element equals $1$. 
However, since $-1<\rho''<0$, the difference $(1+\rho+\rho^2)-1$ is not totally positive. Hence $1+\rho+\rho^2$ is indeed indecomposable. 

Let us now prove the second part of the statement. Let
\[
\delta_1=\frac{1}{a^2+3a+9}(-4-a+(-1-2a)\rho+2\rho^2)(c-k\rho-l\rho^2)
\]
for some $c,k,l\in\Z$. We have
\[
\text{Tr}((1+\rho+\rho^2)\delta_1)=c-(1+a)k-(4+2a+a^2)l,
\] 
and so if this trace equals $1$, then $c=1+k(1+a)+l(4+2a+a^2)$.

Assume now for contradiction that $\delta_1$ is totally positive. We have
\[
\text{Tr}(\delta_1)=-l,
\]
and so $l<0$. Moreover, for $\psi=(a^2+3a+9)\delta_1(1+\rho+\rho^2)\succ 0$ we have
\[
\psi\psi'+\psi'\psi''+\psi''\psi=-(a^2+3a+9)^2 (k^2 + k (l + 2 a l) + l (1 + (1 + a + a^2) l))>0.
\] 
However, this is not possible for any $l<0$ and $k\in\Z$, and so $\delta_1$ is not totally positive, completing the proof.
\end{proof}

Thus $1+\rho+\rho^2$ is qualitatively different from the indecomposables in $\bt$. In particular, this element indicates that  $\min_{\delta\in\O_K^{\vee,+}}\textup{Tr}(\alpha\delta)=1$ does not hold for all indecomposables $\alpha$ in cubic number fields. 
We can also note that the norm of $1+\rho+\rho^2$ is the square root of the discriminant of $\Q(\rho)$, i.e., ${N}(1+\rho+\rho^2)=a^2+3a+9$, and so this element is divisible only by ramified prime ideals.

\subsection{Characterization of indecomposables}

Theorem \ref{thm:main} now immediately follows:

\begin{proof}[Proof of Theorem $\ref{thm:main}$]
In Subsections \ref{subsec:par1} and \ref{subsec:par2} we showed that the only non-unit candidates for indecomposables are $1+\rho+\rho^2$ and $-v-w\rho+(v+1)\rho^2$ where $0\leq v\leq a$ and $v(a+2)+1\leq w\leq (v+1)(a+1)$. Their indecomposability was proved in Lemmas \ref{lemma:indespec} and \ref{lemma:indetria}.
\end{proof}

We can also summarize our results about $\min_{\delta\in\O_K^{\vee,+}}\textup{Tr}(\alpha\delta)$ for an indecomposable $\alpha$ as follows.

\begin{cor} \label{cor:mintracesimpl}
Assume that  $\O_K=\Z[\rho]$ and let $\alpha$ be an indecomposable. Then \[\min_{\delta\in\O_K^{\vee,+}}\textup{Tr}(\alpha\delta)\leq 2.\] 
\end{cor}

Note that the converse implication is not true (at all!) since for example the sum of any two elements from Lemma \ref{lemma:indetria} has this minimal trace also equal to $2$, but is obviously decomposable. Moreover, this Corollary \ref{cor:mintracesimpl} does not hold for general totally real cubic field; we found examples of fields and indecomposable integers for which this minimum is equal to $3$ -- see Subsection \ref{subsec:trace3}.

\section{Number of elements of small norm}\label{sec:small norms}

Let us now use our knowledge of indecomposables to count the primitive elements of small norm up to multiplication by units or, equivalently, such principal ideals, i.e., to prove Theorem~\ref{th count}.
As discussed in the Introduction, there are surprisingly many such small norms!

\medskip

However, let us start by comparing Theorem \ref{th count} to the situation in families of real quadratic fields: 
A natural general class of their families are ``continued fraction families'' \cite{DK} that are obtained by fixing the period length $s$ and all the coefficients $u_1,u_2,\ldots,u_{s-1}$ in the continued fraction expansion
$\xi_D=[u_0,\overline{u_1,u_2,\ldots,u_s}]$.
If the coefficients $u_1,\dots, u_{s-1}$ satisfy certain mild parity condition, then
all the possible $D$s are given by the values of a quadratic polynomial $q(t)$ and there are infinitely many such $D$s that are squarefree \cite{Fr}.
These families generalize most of the well-known 1-parameter families of real quadratic fields such as Chowla's $D=4t^2+1$ or Yokoi's $D=t^2+4$.

In a continued fraction family, the fundamental unit $\alpha_{s-1}$ is  quite small, and one can prove sharp results on the distribution of class numbers \cite{DK,DL}: very roughly, they grow as $\Delta^{1/2}\asymp t$, where $\Delta\asymp t^2$ is the discriminant of the field. 

As we have seen before, the indecomposables in real quadratic fields are known explicitly and, in fact, the only primitive elements of norm $<\sqrt\Delta /2$ are the convergents $\alpha_i,\alpha'_i$ \cite[Proposition 6]{BK}. 
There are thus $\leq 2s$ such elements up to multiplication by units; note that the period length $s$ is constant through the family.

Therefore we see that in a continued fraction family, there are $\asymp 1$ primitive principal ideals of norm $<\sqrt\Delta /2$, which is roughly what we would expect from the prediction that there should be $\Delta^{1/2}/h$ such ideals.

At the moment, we do not know the source of this tantalizing discrepancy between quadratic and cubic fields; nor do we know to what extent it persists in other families in degrees $\geq 3$.

\medskip

Returning to the proof of Theorem \ref{th count},
let us first explain its basic idea before proving the necessary estimates in a series of Lemmas:

The simplest cubic field $K$ has units of all signatures, and so each principal ideal has a totally positive generator $\beta\in\co_K^+$. This element decomposes as a sum of indecomposables $\beta=\sum \alpha_i$. Moreover, we have the following easy consequence of H\" older inequality:

\begin{lemma}[{\cite[Lemma 2.1]{KY}}]\label{lemma:norm_bound}
For all $\alpha_1,\dots,\alpha_k\in \co_K^+$ we have \[{N}\left( \sum_{i=1}^k \alpha_i\right) ^{1/d}\ge \sum_{i=1}^k{N}(\alpha_i)^{1/d}.\]
\end{lemma}

In particular, if  $\beta=\sum\alpha_i$, then $N(\beta)\geq N(\alpha_i)$ for each $i$. Thus as the first step
we will study indecomposables of norm $\leq a^2$, and then we will consider which of their sums still yield elements of small norm.

Note that this strategy is fairly robust: it does not require the monogenicity of $\co_K$ nor the existence of units of all signatures. It is ``only'' necessary to know the indecomposable elements of all the semirings $\co_K^{\sigma}$ of elements of signature $\sigma$ for all signatures $\sigma$; note that Lemma \ref{lemma:norm_bound} holds for elements of fixed signature $\sigma$ with essentially the same proof, one only has to take the absolute values of the norms in the statement.

\subsection{Sizes of units} In this Subsection, we derive several properties of (totally positive) units of the simplest cubic fields that we will soon need. 
The first of them tells us that at least one conjugate of each unit except for $1$ is large.

\begin{lemma} \label{lemma:units>a^2}
	Let $a\geq 7$ and let $\varepsilon$ be a unit such that $|\varepsilon|,|\varepsilon'|,|\varepsilon''|<a$. Then $\varepsilon=1$.
\end{lemma}

\begin{proof}
	If $a\geq 7$, we have $a+1<\rho<a+1+\frac{2}{a}$, $1+\frac{1}{2a}<|\rho'|<1+\frac{1}{a}$, and $\frac{1}{a+3}<|\rho''|<\frac{1}{a+2}$ \cite{LP}. 
	
	Assume that $\varepsilon=\rho^k\rho'^l$ (for some $k,l\in\Z$) satisfies the assumption. If $k\geq 1$, then $\rho^k>(a+1)^k$ and necessarily $l\leq -1$ to get $|\varepsilon|<a$. However, for $\varepsilon'=\rho'^k\rho''^l$, we have $|\rho'|^k>1$ for $k\geq 1$, and $|\rho''|^l>a+2$ for $l\leq -1$, i.e., $|\varepsilon'|>a+2$. Hence our unit cannot have $k\geq 1$.
	
	If $k\leq -1$, then $|\rho''|^k>(a+2)^{-k}$, and necessarily $l\leq k$ to get \[a>|\varepsilon''|=|\rho''|^k\rho^l>(a+2)^{-k}\left(a+1+\frac{2}{a}\right)^l.\] However, we also obtain
	\[
	|\varepsilon'|=|\rho'|^k|\rho''|^l>\left(\frac{a}{a+1}\right)^{-k}(a+2)^{-l}>a
	\]
	for $l\leq k\leq -1$. Thus the choice $k\leq -1$ does not give any suitable $\varepsilon$ and we are left only with the case $k=0$.
	
	If $k=0$, then $\ve=\rho'^l$ which is the conjugate of $\rho^l$. We can therefore apply the previous part of the proof again to conclude that $l=0$ and $\ve=1$.
\end{proof}

In particular, the previous Lemma tells us that any totally positive unit (as a square) is greater than $a^2$ in some embedding. 
However, in some cases, we will need to know more.

\begin{lemma} \label{lemma:a^4/a^2+a^2}
	Let $a\geq 7$ and let $\ve$ be a totally positive unit such that $\ve>a^2$. If $\ve\neq \rho^2,\rho''^{-2}$, then at least one of the following holds:
	\begin{enumerate}
		\item $\ve>a^4$, or
		\item $\ve'>a^2$, or
		\item $\ve''>a^2$.
	\end{enumerate} 
\end{lemma}

\begin{proof}
	Let $\ve=\rho^k\rho'^l$, where $k,l$ are even so that $\ve\succ 0$. Since $a+1<\rho<a+1+\frac{2}{a}$ and $-1<\rho''<2$, we must have $k\geq 2$ to get $\ve>a^2$. If $l<0$, then $\ve'>(1+\frac{1}{2a})^k(a+2)^{-l}>a^2$, i.e., $\ve'>a^2$. Thus it remains to consider the case $l\geq 0$.
	
	If moreover $k\geq 4$, then $\ve>(a+1)^k(1+\frac{1}{2a})^l>a^4$. Thus we are left with the case $k=2$. 
	
	If $l\geq 6$, then $\ve''>\frac{(a+1)^l}{(a+3)^2}>a^2$. 
	
	If $l=4$, we can consider the monic minimal polynomial $f$ of $\ve=\rho^2\rho'^4$, for which one computes that $f(a^2)>0$ (if $a\geq 1$). Thus this polynomial has 0 or 2 roots that are $>a^2$.
	Since $\rho^2\rho'^4>a^2$ is its root, $f$ must have another large root, i.e., $\rho^2\rho'^4$ has some conjugate $>a^2$ 
	
	The remaining cases $l=0,2$ give the units $\rho^2$ and $\rho''^{-2}$ which are excluded in the statement.    
\end{proof}

\subsection{Estimates on norms}

\begin{lemma} \label{lemma:norm_ineq}
	Let $a\geq 3$ and assume that $\alpha(v+1,W)\in\bt_0$. Then $N(\alpha(v,W))<N(\alpha(v+1,W))$.
\end{lemma}

\begin{proof} Fix $v$ and consider the difference $f(W)=N(\alpha(v+1,W))-N(\alpha(v,W))$ as a quadratic polynomial in $W$. The leading coefficient of this polynomial is negative, and so the set of $W$s at which the polynomial has positive value is convex. Thus it suffices to prove that $f(W)>0$ for the endpoints
$W=v+1$ and $W=3A+a_0-2v-3$, which is a straightforward calculation.	One also directly checks the case $\alpha(A,A)$ (if $3\mid a$).
\end{proof}

\begin{lemma}\label{l:first row} Let $\delta\in [0,1]$ and assume that $a\geq 3$.
	If an indecomposable $\alpha$ has norm $\leq a^{1+\delta}$, then, up to multiplication by units and conjugation, $\alpha=\alpha_w=-w\rho+\rho^2$ for some $0\leq w< a^{\delta /2}$.	
\end{lemma}

\begin{proof}
	We need to consider only indecomposables $-v-w\rho+(v+1)\rho^2$ (where $0\leq v\leq a$ and $v(a+2)+1\leq w\leq (v+1)(a+1)$) from the triangle from Lemma \ref{lemma:indetria}, as the unit $1$ is associated to $\alpha_0=\rho^2$ and $1+\rho+\rho^2$ has norm $>a^2$. By conjugating our indecomposable if necessary, we can restrict only to elements from $\bt_0$ for which we have Lemma \ref{lemma:norm_ineq}.
	
	First we prove that if $v>0$ and our element belongs to $\bt_0$, then $N(-v-w\rho+(v+1)\rho^2)>a^2$. By Lemma \ref{lemma:norm_ineq}, it suffices to show this only for $v=1$. In that case, the norm increases with increasing $w$ and then eventually starts to decrease. Thus, it is enough to prove that the norm at both endpoints $-1-((a+2)+2)\rho+2\rho^2$ and $-1-((a+2)+1+(a-3))\rho+2\rho^2$ is greater than $a^2$. These norms are $2a^2+6a-9$ and $4a^2-17$, both of which are large enough for $a\geq 3$.
	
	Thus we are left to deal with the elements $\alpha_w=-w\rho+\rho^2$ whose norms are $N(\alpha_w)=-w^3+aw^2+(a+3)w+1$. 
	Again, for all real $w\geq 0$ this function is first increasing and then decreasing, and so it suffices to check its values at the endpoints of the interval $[a^{\delta /2},a]$ that we want to exclude. Both of them 
	are indeed $>a^{1+\delta}$.
\end{proof}

The key case is now to consider the sums of the form $\sum\alpha_{w_i}$.

\begin{proposition}\label{pr:sum alpha} Let $\delta\in [0,1]$ and
	let $\mathcal Q_\delta$ be the number of primitive elements of the form $\beta=\sum_{i=1}^k\alpha_{w_i}$ with $k\in\Z_{\geq 1}$ and $0\leq w_i<a^{\delta /2}$, that have $N(\beta)\leq a^{1+\delta}$. 
	Then $\mathcal Q_\delta \asymp a^{2\delta/3}$.	
\end{proposition}

\begin{proof}
Let us start by showing the upper bound. We have $\sum_{i=1}^k\alpha_{w_i}=k\alpha_{w/k}=k(-(w/k)\rho+\rho^2)$ where 
$w=\sum_{i=1}^k{w_i}$ (of course, $w/k$ need not be an integer), and so 
$N(\beta)=-w^3+akw^2+ak^2w+3k^2w+k^3\leq a^{1+\delta}$. Note that the element $\beta=\sum_{i=1}^k\alpha_{w_i}$ does not depend on the individual $w_i$s, but only on $k$ and $w=\sum w_i$.

As $\beta$ is primitive, we have $w>0$ (with the exception of the trivial case $k=1, w=0$).

We have $w=\sum w_i<ka^{\delta/2}< ak$. 
Thus 
$ak^2w< N(\beta)\leq a^{1+\delta}$, which in turn implies $k^2w< a^{\delta}$.
Similarly we prove that $kw^2< a^{\delta}$.

Let us now distinguish the cases $k\geq w$ and $k<w$. In the first case, we use $k^2w< a^{\delta}$ to see that $w<a^{\delta/3}$ and we
estimate the number of such pairs of positive integers as
\begin{align*}
\#\{(k,w)\mid k^2w< a^{\delta},w<a^{\delta/3} \}=\sum_{w=1}^{a^{\delta/3}}\ \sum_{k=1}^{a^{\delta/2}w^{-1/2}}1&\leq
a^{\delta/2}\sum_{w=1}^{a^{\delta/3}}w^{-1/2}
\ll a^{\delta/2}\cdot a^{\delta/6}= a^{2\delta/3}.
\end{align*}
In the second case $k<w$ we analogously use $kw^2<a^{\delta}$ to again deduce that the number of such pairs $\ll a^{2\delta/3}$, finishing the proof of the upper bound.

\medskip

For the lower bound, consider the pairs with coprime $k\asymp a^{\delta/3}, w\asymp a^{\delta/3}$: for concreteness, we can take $0.01a^{\delta/3}<k,w<0.1a^{\delta/3}$. There are $\gg a^{2\delta/3}$ such pairs and each of them gives an element with norm $<a^{1+\delta}$. However, we still need to make sure that different pairs $(k,w)$ give different elements $k\alpha_{w/k}$ (up to conjugation and multiplication by units) and that these elements are primitive.

Consider the element $\gamma=k\alpha_{w/k}\rho^{-1}=-w+k\rho$. From the estimates of the sizes of the conjugates of $\rho$ we see that  
$$a^{4\delta/3}\asymp ka\asymp\gamma,|\gamma'|\asymp k+w\asymp a^{\delta/3},|\gamma''|\asymp w\asymp a^{\delta/3},$$
Thus the largest conjugate is $\gamma$; for its size we more precisely have
$$k(a+1)-w<-w+k\rho<k(a+1)-w+\frac{2k}a,$$
which determines $k$ uniquely.
The uniqueness of $w$ then follows immediately.

Moreover, as each totally positive unit $\neq 1$ has a conjugate $> a^2$ by Lemma \ref{lemma:units>a^2}, it is not possible that multiplying $\gamma,\gamma',$ or $\gamma''$ by a unit would produce another element of the same shape, but with different values of $(k,w)$.

The element $k\alpha_{w/k}=-w\rho+k\rho^2$ is primitive if and only if $k$ and $w$ are coprime, as $1,\rho,\rho^2$ is an integral basis for $\co_K$.
\end{proof}

To finish the proof of Theorem \ref{th count}, it now suffices to show that if $\beta=\sum \alpha_i$ contains any other sum of indecomposables than the one in Proposition \ref{pr:sum alpha}, then it has norm $>a^2$. Hence we have to consider the other sums of indecomposable integers of norm $<a^2$; in these sums can appear totally positive units, indecomposables $\alpha_w$, their conjugates, and the unit multiples of them or of their conjugates. Except for the cases covered by Proposition \ref{pr:sum alpha}, we will show that any sum of two such elements has norm $>a^2$.

In this task, we will use two different tools. First of all, we have the relation 
\begin{equation} \label{eq:nsrel}
N(\alpha+\beta)=N(\alpha)+N(\beta)+\Tr(\alpha\beta'\beta'')+\Tr(\alpha\alpha'\beta''). 
\end{equation}
In the following proofs, we often show that some summand in $\Tr(\alpha\beta'\beta'')$ or $\Tr(\alpha\alpha'\beta'')$ is $>a^2$. On the other hand, we can sometimes express our norm explicitly and make some comparisons to show that it is really $>a^2$.

Moreover, we can restrict our attention to a few specific cases. For example, any element of the form $\ve_1\alpha_w+\ve_2$, where $\ve_1,\ve_2$ are totally positive units, has the same norm as $\alpha_w+\ve_2\ve_1^{-1}$, so it suffices to consider only the case $\alpha_w+\ve$. This also covers the case when $\alpha_w$ is replaced by one of its conjugates. We can make similar consideration for all the other cases and come to the conclusion that it suffices to discuss the sums 
$1+\ve$, $\alpha_w+\ve$, $\alpha_w+\ve\alpha_t$, and $\alpha_w+\ve(\alpha_t)'$, where $\ve$ is an arbitrary totally positive unit and $w\neq 0$. 
By Lemma \ref{l:first row} we can always assume $1\leq w<\sqrt a$.
Note that we can also exclude the cases when our sum is not primitive. In the following, we assume $a\geq 8$.

\subsubsection{The case $1+\ve$}
In this case, we have $\ve\neq 1$ since otherwise we get $2$, which is not primitive. Thus Lemma \ref{lemma:units>a^2} gives us that one conjugate of $\ve$ is greater than $>a^2$. Without loss of generality, we can assume that $\ve>a^2$. Using \eqref{eq:nsrel} for $\alpha=1,\beta=\ve$ we see that $\alpha'\alpha''\beta=\ve>a^2$ which is a summand in $\Tr(\alpha\alpha'\beta'')$. Thus necessarily $N(1+\ve)>a^2$.   

\subsubsection{The case $\alpha_w+\ve$}

First of all, let us discuss the case when $\ve=1$. Then
\[
N(\alpha_w+1)=-w^3-3w^2+a^2w+3aw+6w+2a^2+6a+17>a^2
\]
for all $1\leq w< \sqrt{a}$.

Let now $\ve\neq 1$. Lemma \ref{lemma:units>a^2} gives us that $\ve$ is $>a^2$ in some embedding. Moreover, we can easily show that for $a\geq 7$,
\begin{align*}
-w\rho+\rho^2&>a^2+2a+1-w(a+2),\\
-w\rho'+\rho'^2&>w+1,\\
-w\rho''+\rho''^2&>\frac{w}{a+3}.
\end{align*}
Furthermore, we can see that $a^2+2a+1-w(a+2)>a+3$ for all $1\leq w< \sqrt{a}$.

If $\ve'>a^2$, then $\alpha_w(\alpha_w)''\ve'>a^2$, and if $\ve''>a^2$, then $\alpha_w(\alpha_w)'\ve''>a^2$, which for both of these cases implies that the considered norm is $>a^2$. 

However, we cannot straightforwardly make the same statement for the case when $\ve>a^2$. Lemma \ref{lemma:a^4/a^2+a^2} gives us that for almost all the totally positive units, either $\ve>a^4$ (in which case $(\alpha_w)'(\alpha_w)''\ve>a^2$), or $\ve'$ or $\ve''$ is greater than $a^2$, so we can use the previous lines to conclude that this norm is $>a^2$. So we are left with units $\rho^2$ and $\rho''^{-2}$. The case when $\ve=\rho^2$ is included in Proposition \ref{pr:sum alpha}. On the other hand, using an explicit expression for $N(\alpha_w+\rho''^{-2})$, it can be shown that it is $>a^2$.

\subsubsection{The case $\alpha_w+\ve\alpha_t$} 

The choice $\ve=1$ is included in Proposition \ref{pr:sum alpha}, so we do not need to consider it here. 

In the other cases, we have 
\begin{align*}
(\alpha_w)'(\alpha_w)''\ve\alpha_t>a^2& \quad \text{ if } \ve>a^2, \\
\alpha_w(\alpha_w)''(\ve\alpha_t)'>a^2& \quad \text{ if } \ve'>a^2, \\ 
\alpha_w(\alpha_w)'(\ve\alpha_t)''>a^2& \quad \text{ if } \ve''>a^2. 
\end{align*}

Thus the norm of $\alpha_w+\ve\alpha_t$ is always too large.

\subsubsection{The case $\alpha_w+\ve(\alpha_t)'$}
If $\ve=1$, it can be easily shown that the summand $\alpha_w(\alpha_w)'(\alpha_t)''$ in $N(\alpha_w+(\alpha_t)')$ is greater than $a^2$, and thus this norm is $>a^2$.

If $\ve\neq 1$, we can directly decide the following cases:
\begin{enumerate}
\item If $\ve''>a^2$, then $\alpha_w(\alpha_w)'(\ve\alpha_t)''>a^2$.
\item If $\ve>a^4$, then $(\alpha_w)'(\alpha_w)''\ve\alpha_t>a^2$.
\item If $\ve'>a^4$, then $\alpha_w(\alpha_w)''(\ve\alpha_t)'>a^2$.
\item If $\ve>a^2$ and $\ve'>a^2$, then $(\alpha_w)''\ve\alpha_t(\ve\alpha_t)'>a^2$. (Here we use the facts that $t+1\geq 2$ and $2a^2>(a+3)^2$ for $a\geq 8$.)
\end{enumerate}  
Thus it remains to discuss the cases when $\ve$ or $\ve'$ are equal to $\rho^2$ or $\rho''^{-2}$. Using the explicit expression for these norms, it can again be shown that all are $>a^2$.
   
\medskip

We showed that in all cases except for Proposition \ref{pr:sum alpha}, the considered sums have norm $>a^2$. This finishes the proof of Theorem \ref{th count}.

\bigskip

Note that similarly, one can also use such estimates to determine the largest norm of an indecomposable, following the numerous results for real quadratic fields \cite{DS,HK,JK,Ka2,TV}. One obtains that if $a\geq 4$ then the largest norm is $\sim a^4/27$ and corresponds to the indecomposable nearest to the center of the triangle $\bt$
(the details are quite technical, and so we leave them to a follow-up article).

Further, let us remark that Valentin Blomer [personal communication] can use Theorem \ref{th count} to prove that a
positive density of indecomposables from the triangle $\bt$ have squarefree norms. This is often useful for dealing with universal forms; cf. Table \ref{table sqfree}.

An interesting problem is to see how far one can push this method: e.g., can one also count elements up to norm $<a^4\sim\Delta$? We do not see any fundamental obstacle to this, except for the fact that the number of decompositions that need to be considered quickly increases with the norm.

\section{Universal forms over simplest cubic fields} \label{sec:qforms}

Now we can apply our results on indecomposables to universal quadratic forms. 

Let $K$ be a totally real number field of degree $d$ over $\Q$ and $\co$ an order in $K$.
Let us consider a quadratic form $Q$ of rank $r$ over $\co$, i.e., $Q(x)=\sum_{1\leq i\leq j\leq r}a_{ij}x_ix_j$ with $a_{ij}\in\co$. 
Such a form $Q$ is \emph{totally positive}  if $Q(v)\succ 0$ for all $v\in\co^r, v\neq 0$. The form $Q$ is \textit{diagonal} if $a_{ij}=0$ for all $i\neq j$, and \textit{classical} (or \textit{classically integral}) if $2\mid a_{ij}$ in $\co$ for all $i\neq j$.

The quadratic form $Q$ \textit{represents} an element $\alpha\in\co^+$ over the order $\co$ if  $Q(v)=\alpha$ for some $v\in \co^r$. We say that $Q$ is \emph{universal over $\co$} if it is totally positive and represents every element  $\alpha\in\co^+$ over $\co$. When dealing with the maximal order $\co_K$, we often just say that $Q$ is universal (or universal over $K$ to specify the number field).

\medskip

We will work in the language of quadratic lattices in some of the proofs below, so let us briefly introduce it:
A \emph{totally positive quadratic space over $K$} is an $r$-dimensional vector space $V$ over the field $K$ equipped with a symmetric bilinear form $B: V\times V\to K$ such that the associated quadratic form $Q(v)=B(v,v)$ satisfies $Q(v)\succ 0$ for all non-zero $v\in V$. 
A \emph{quadratic $\co$-lattice} $L\subset V$ is an $\co$-submodule such that $KL=V$; $L$ is equipped with the restricted quadratic form $Q$, and so we often talk about a quadratic $\co$-lattice as the pair $(L,Q)$.
A \emph{sublattice} of the lattice $(L,Q)$ is an $\co$-submodule equipped with the restriction of the quadratic form $Q$.
Note that to a totally positive quadratic form $Q$ over $\co$ we can naturally associate the
quadratic $\co$-lattice $(\co^r,Q)$, and so we will interchangeably talk of $Q$ as a quadratic form and as an $\co$-lattice.

\subsection{General results}

Let us first prove several general Propositions that relate indecomposables and elements of small trace to ranks of universal forms.

Recall that the Pythagoras number of a ring $R$ is defined as the smallest integer $s(R)$ such that 
every sum of squares of elements of $R$ can be expressed as the sum of $s(R)$ squares. It is well-known that the Pythagoras number of an order $\co$ in a totally real number field $K$ is always finite (although it can be arbitrarily large \cite{Sc2}) and that $s(\co)\leq f(d)$ where $f(d)$ is a function that depends only on the degree $d$ of $K$ \cite[Corollary 3.3]{KY}.
Moreover, one can take $f(d)=d+3$ for $d=2,3,4,5$.

\pagebreak

\begin{proposition}\label{pr:univ construct}
	Let $K$ be a totally real number field and $\co$ an order in $K$. Let $s=s(\co)$ be the Pythagoras number of $\co$.
	Let $\mathcal S$ denote a set of representatives of classes of indecomposables in $\co$ up to multiplication by squares of units $\co^{\times 2}$.
	Then the diagonal quadratic form 
	$$\sum_{\sigma\in\mathcal S}\sigma\left(x^2_{1,\sigma}+x^2_{2,\sigma}+\dots+x^2_{s,\sigma} \right)$$
	is universal over $\co$ (and has rank $s\cdot\#\mathcal S$).
\end{proposition}

\begin{proof}
	Every $\alpha\in\co^+$ can be written as the sum of indecomposables (typically not in a unique way). Combining indecomposables that are the same modulo $\co^{\times 2}$, we obtain 
	$$\alpha=\sum_{\sigma\in\mathcal S}\sigma\beta_\sigma,$$
	where each $\beta_\sigma$ is the sum of squares of units.
	
	As the Pythagoras number is $s$, the sum of squares
    $\beta_\sigma$ is represented by $x^2_{1,\sigma}+x^2_{2,\sigma}+\dots+x^2_{s,\sigma}$.
\end{proof}

The proofs of the following three Propositions are essentially due to Yatsyna (cf. \cite[Corollary 25]{Ya}), but let us include them here for completeness. In all of them, we have the following set-up:

\medskip

Let $K$ be a totally real number field of degree $d$ over $\Q$ and $\co$ an order in $K$. 
Assume that there are
\begin{itemize}
	\item $\delta\in\co^{\vee,+}$,
	\item $n$ elements $\alpha_1,\dots,\alpha_n\in\co^+$ such that $\Tr(\delta\alpha_i)=1$ for each $i$,
	\item $m$ indecomposable elements $\beta_1,\dots,\beta_m\in\co^+$ such that $\Tr(\delta\beta_i)=2$ for each $i$.
\end{itemize}

\begin{proposition}\label{pr:univ upper class}
	Every classical universal quadratic form over $\co$ has rank at least $n/d$.
\end{proposition}

\begin{proof} Let $Q$ be a classical universal quadratic form of rank $r$, and let $v_i\in\co^r$ be such that $Q(v_i)=\alpha_i$.
	
	Consider the quadratic form $q(x_{11},\dots,x_{rd})=\Tr(\delta Q(x_1,\dots,x_r))$. This is a positive definite classical quadratic form over $\Z$ of rank $rd$, as can be seen by expressing each $x_i=\sum_{j=1}^d x_{i1}\omega_1+\dots+x_{id}\omega_d$ as a $\Z$-linear combination of a fixed integral basis $\omega_1,\dots,\omega_d$ for $\co$.
	
	To the vector $v_i$ corresponds a vector $w_i\in\Z^{rd}$ such that 
	$$q(w_i)=\Tr(\delta Q(v_i))=\Tr(\delta \alpha_i)=1.$$
	
	Thus each $w_i$ is a minimal vector of norm 1 of the quadratic lattice $L$ corresponding to $q$. 
	The vectors $w_i$ are pairwise orthogonal,
	as by Cauchy-Schwarz inequality, we have $1=q(w_i)q(w_j)\geq b(w_i,w_j)^2$, and  so (since $b(w_i,w_j)\in\Z$) we have $b(w_i,w_j)=0$ (i.e., $w_i$ is orthogonal to $w_j$) or $b(w_i,w_j)=\pm 1$ (i.e., $w_i=\pm w_j$ -- but this is not possible as $\alpha_i=Q(v_i)\neq Q(v_j)=\alpha_j$).

	Thus the number $n$ of these pairwise orthogonal vectors is less or equal to the rank of $q$, i.e., $n\leq rd$.
\end{proof}

Using the number $m$ of trace two elements one can obtain good estimates at least for diagonal forms.

\begin{proposition}\label{pr:univ upper diag}
	Every diagonal universal quadratic form over $\co$ has rank at least $m/M(d)$, where $M(d)$ is defined in Table $\ref{table lattices}$ and satisfies $M(d)\leq\max(240,d(d-1))$.
\end{proposition}

\begin{proof}
	Let $Q(x_1,\dots,x_r)=\sum \gamma_jx_j^2$ be a diagonal universal quadratic form of rank $r$.
	As all the elements $\beta_i$ are indecomposable, each of them is represented by some subform $\gamma_jx_j^2$.
	
	As in the proof of Proposition \ref{pr:univ upper class}, consider the quadratic form $$q(x_{11},\dots,x_{rd})=\Tr(\delta Q(x_1,\dots,x_r))=\sum_{j=1}^r \Tr(\delta\gamma_jx_j^2).$$
	
	Each of the classical forms $q_j(x_j)=\Tr(\delta\gamma_jx_j^2)$ is $d$-ary; 
	consider its sublattice spanned by the vectors that correspond to the representations of (some of) our elements
	$\beta_i$. This is a classical lattice of rank $\leq d$ with minimal vectors of norm 2. 
	
	Such a lattice is the direct sum of certain root lattices \cite[Theorem 4.10.6]{martinet}, and so we see from \cite[Table 4.10.13]{martinet} that the maximal possible number $2M(R)$ of norm 2 minimal vectors in a rank $R$ lattice is given by Table \ref{table lattices}. 
	(The situation is slightly more complicated in the omitted ranks $9\leq R\leq 15$, as one needs to take a suitable sum of these root lattices.)

\begin{table}[h]
	
	\caption{Numbers of minimal vectors\label{table lattices}}
		\begin{tabular}{|c|c|c|c|c|c|c|c|c|c|c|}
		\hline
		$R$ & 1 & 2 & 3 & 4  & 5  &  6 &  7 & 8 & $\dots$ & $16\leq R$ \\
		\hline
		$M(R)$ & 1 & 3 & 6 & 12 & 20 & 36 & 63 & 120 & $\dots$  & $R(R-1)$ \\
		\hline
		lattice &$A_1$& $A_2$& $A_3$& $D_4$& $D_5$& $E_6$& $E_7$& $E_8$&$\dots$& $D_R$\\
		\hline
	\end{tabular}
	
\end{table}

	To each element $\beta_i$ (that is represented by the subform $\gamma_jx_j^2$) correspond two minimal vectors $\pm w_i$ for the lattice $q_j$ of rank $\leq d$, and so 
	$\gamma_jx_j^2$ can represent at most $M(d)$ of our elements $\beta_i$. Therefore we have $rM(d)\geq m$, as we needed to prove.
\end{proof}

Similarly we can estimate the ranks of non-classical forms in terms of the number $n$ of trace one elements: 

\begin{proposition}\label{pr:univ upper noncl}
	If $n\geq 240$, then
	every (non-classical) universal quadratic form over $\co$ has rank at least $\sqrt n/d$.
\end{proposition}

\begin{proof}
	If $Q$ is a (non-classical) universal form of rank $r$, then $2\Tr(\delta Q)$ is a classical form of rank $rd$ whose minimal vectors have norm 2. 
	We can now argue in the same way as in the proof of Proposition \ref{pr:univ upper diag} and note that we have $M(rd)\geq n\geq 240$ to conclude that the rank $rd\geq 16$. Thus $(rd)^2>rd(rd-1)\geq M(rd)\geq n$, as needed.
\end{proof}

Note that in the last Proposition, if $n<240$, then we can use the estimate that $M(R)<2R^2$ for all $R$ to conclude that the rank must be $>\frac{\sqrt n}{d\sqrt 2}$.

\subsection{Simplest cubic fields}

In this Subsection, let $K$ be the simplest cubic field (with the parameter $a$) such that its ring of integers $\co_K=\Z[\rho]$ (this is, e.g., true if the square root of the discriminant $\sqrt\Delta=a^2+3a+9$ is squarefree).

In Theorem \ref{thm:main} we explicitly described the set $\mathcal S$ of indecomposables (modulo squares of units); we see that $$\#\mathcal S=\frac{a^2+3a+6}2.$$

As the Pythagoras number $s$ in degree $d=3$ satisfies $s\leq 6$, we conclude  from Proposition \ref{pr:univ construct} that there is a diagonal universal form of rank $6\#\mathcal S=3(a^2+3a+6)$.

\medskip

In Lemma \ref{lemma:indetria} we showed that all the elements from the triangle of indecomposables have trace 1 (after multiplication by  suitable $\delta\in\co_K^{\vee,+}$). Further, we also saw in Subsection \ref{subsec:triangle} that the units $\rho^2=\alpha(0,0),\alpha(-1,a+1),\alpha(a+1,0)$ also give trace 1 after multiplication by this $\delta$.

Thus we can take $$n=\frac{a^2+3a+8}2.$$ 
By Propositions \ref{pr:univ upper class} and \ref{pr:univ upper noncl} this concludes the proof of Theorem \ref{thm:main univ}.

\medskip

Note that one can get a better constant for the lower bound on the ranks of diagonal universal forms by the following argument: 

Let $Q(x_1,\dots,x_r)=\sum \gamma_jx_j^2$ be a diagonal universal quadratic form of rank $r$.
Let $\alpha$ be an indecomposable that is squarefree in the sense that $\beta^2\mid \alpha$ implies $\beta\in\co_K^{\times}$. Then $\alpha$ equals $\gamma_j$ (up to multiplication by a unit) for some $j$. Thus the number of squarefree indecomposables up to units gives a lower bound on the rank of diagonal universal forms. 

In Table \ref{table sqfree} we computed the number $\text{sq}(a)$ of indecomposables (up to multiplication by units) that have squarefree norm.
Note that in total there are $\sim a^2/2$ indecomposables, and so vast majority of them indeed have squarefree norm.

\begin{center}
	
\begin{table}[h]
	\caption{The number of non-associated indecomposables with squarefree norm.\label{table sqfree}}
	
	\begin{tabular}{|c|c|}
		\hline
		$a$ & $\text{sq}(a)$\\
		\hline
		-1 & 2 \\ 
		
		0 & 2 \\ 
		
		1 & 5 \\ 
		
		2 & 8 \\ 
		
		4 & 17 \\ 
		
		6 & 22 \\ 
		
		7 & 38 \\ 
		
		8 & 47 \\ 
		
		9 & 46 \\ 
		
		10 & 68 \\ 
		
		11 & 59 \\ 
		\hline
	\end{tabular}
	\begin{tabular}{|c|c|}
		\hline
		$a$ & $\text{sq}(a)$ \\
		\hline
		
		13 & 101 \\ 
		
		14 & 122 \\ 
		
		15 & 118 \\ 
		
		16 & 110 \\ 
		
		17 & 158 \\ 
		
		18 & 166 \\ 
		
		19 & 209 \\ 
		
		20 & 224 \\ 
		
		22 & 272 \\ 
		
		23 & 272 \\ 
		
		24 & 265 \\
		\hline
	\end{tabular}
	\begin{tabular}{|c|c|}
		\hline
		$a$ & $\text{sq}(a)$ \\
		\hline
		
		25 & 341 \\ 
		
		26 & 275 \\ 
		
		27 & 346 \\ 
		
		28 & 404 \\ 
		
		29 & 455 \\ 
		
		31 & 404 \\ 
		
		32 & 539 \\ 
		
		33 & 517 \\ 
		
		34 & 593 \\ 
		
		35 & 614 \\ 
		
		36 & 496 \\ 
		\hline
	\end{tabular}
	\begin{tabular}{|c|c|}
		\hline
		$a$ & $\text{sq}(a)$\\
		\hline
		
		37 & 575 \\ 
		
		38 & 755 \\ 
		
		40 & 839 \\ 
		
		42 & 811 \\ 
		
		43 & 983 \\ 
		
		44 & 884 \\ 
		
		45 & 928 \\ 
		
		46 & 833 \\ 
		
		47 & 1157 \\ 
		
		49 & 1277 \\ 
		
		50 & 1166 \\ 
		\hline
	\end{tabular}
\end{table}

\end{center}

\subsection{Real quadratic fields}\label{subsec:7.3}

For comparison, let us describe the corresponding results for real quadratic fields (using the notation introduced in Section \ref{sec:quadr}).

By Proposition \ref{pr:quad} we have that $n=\max\{u_i\mid 1\leq i<s\text{ odd}\}+1$ if $s$ is even, and $n=\max\{u_i\mid 1\leq i<s\}+1=2u_{s-1}+1\asymp\sqrt\Delta$ if $s$ is odd, whereas the number of indecomposables is
\[
\#\mathcal S=\sum_{j=1}^s u_{2j-1}
=\begin{cases}
2(u_1+u_3+u_5+\dots+u_{s-1})&\mbox{if } s\mbox{ is even,}\\
2u_0+u_1+u_2+\dots+u_{s-1}&\mbox{if } s\mbox{ is odd}.
\end{cases}
\]
Blomer and Kala \cite[Theorem 2]{BK2} proved that $\#\mathcal S\ll \sqrt \Delta(\log \Delta)^2$, but getting better estimates is quite a subtle question.
We can take $5$ as an upper estimate on the Pythagoras number $s$.

While our upper bound from Proposition \ref{pr:univ construct} is essentially the same as \cite[Theorem 10]{BK2},
the lower bound from Proposition \ref{pr:univ upper class} is typically significantly stronger than the results from \cite{BK2}, as it deals not only with diagonal forms, and even often provides a better bound.

Nevertheless, typically the lower bound $n/2$ is not of the same order of magnitude as the upper bound $5\#\mathcal S$.

\medskip

In certain very special families, such as $D=\sqrt {t^2+1}$ with odd $t$, these estimates are quite sharp:
We have $\sqrt D=[t,\overline{2t}]$, and so $n=2t+1$ and $\#\mathcal S=2t$. Thus we for example have
that the smallest rank $r$ of a classical universal form satisfies $t+1\leq r\leq 10t$.

\section{Other families}\label{sec:other}

For comparison, in this Section we provide similar results for two families of non-Galois cubic fields that do not possess units of all signatures.
First we consider Ennola's family in which there are only $a\asymp \Delta^{1/4}$ indecomposables, most of which appear to have minimal trace 2. 
Even stranger is a family that was considered by Thomas, in which we obtain an indecomposable of minimal trace 3.
In both cases we obtain lower and upper bounds on the ranks of diagonal universal forms that are of magnitude $a\asymp \Delta^{1/4}$. It seems also possible to obtain an analogue of Theorem \ref{th count}, although this would require also determining indecomposables in other signatures, which we have not done (yet).

We provide only very brief sketches of the arguments; the full details will appear in a forthcoming paper.

\subsection{Ennola's cubic fields} \label{subsec:8.1}

A well-studied family are Ennola's cubic fields \cite{En,En1}, i.e., the fields generated by a root $\rho$ of the polynomial $f(x)=x^3+(a-1)x^2-ax-1$ with $a\geq 3$. The discriminant of this polynomial is
$$
a^4+6a^3+7a^2-6a-31,
$$
which is an irreducible polynomial in $a$. If its value is squarefree, then $\O_K=\Z[\rho]$; however, this can also occur for some other values of the discriminant (e.g., for $a=9$). For example, among $3\leq a \leq 1000$, we found $959$ fields having $\O_K=\Z[\rho]$, although it is an open problem whether there are infinitely many such $a$.
In contrast with the simplest cubic fields, Ennola's fields are not Galois, which somewhat (but not fatally) complicates some of our arguments concerning indecomposables. 

Ennola proved that if the index $[\co_K:\Z[\rho]]\leq \frac{a}{3}$, then $\rho$ and $\rho-1$ are fundamental units of $\co_K$. In particular, this is true if $\co_K=\Z[\rho]$. 

Since we do not know that $\co_K=\Z[\rho]$ happens infinitely many times, we will state our results for the order $\Z[\rho]$. Thomas \cite{Th} established that the units $\rho$ and $\rho-1$ form the system of fundamental units of $\Z[\rho]$; for more information see also \cite{Lo2,Lo3}. Note that $\rho$ and $\rho-1$ have the same signature and are not totally positive. Thus $\rho^2$ and $\rho(\rho-1)$ generate the group of totally positive units.

The first parallelepiped we consider is generated by the totally positive units $1$, $\rho^2$, and $\rho(\rho-1)$ and does not contain any interior element. On the other hand, the parallelepiped generated by $1$, $\rho^2$, and $\rho(\rho-1)^{-1}=1+a\rho+\rho^2$ 
produces non-unit elements of the form $1+w\rho+\rho^2$ where $1\leq w\leq a-1$. Now we will study them in more detail.

Let us consider $\delta=\frac{-(a-1)+(a-1)\rho+\rho^2}{f'(\rho)}\in\Z[\rho]^{\vee,+}$. It can be easily shown that
$
\Tr(\delta(v_1+v_2\rho+v_3\rho^2))=v_1+v_3.
$
Thus $\Tr(\delta(1+w\rho+\rho^2))=2$. The proof of the indecomposability of our elements is then analogous to the proof for the exceptional element in Subsection \ref{subsec:except}. 
If $1+w\rho+\rho^2$ were decomposable, it could be written as the sum of two totally positive elements of the form $v+t\rho+(1-v)\rho^2$. One of them would have to have $v\geq 1$ -- but the only such totally positive element is $1$. But then $1+w\rho+\rho^2-1$ is not totally positive.
One thus establishes:

\begin{proposition}
Let $\rho$ be a root of the polynomial $x^3+(a-1)x^2-ax-1$ where $a\geq 3$. Then $1$ and $1+w\rho+\rho^2$ where $1\leq w\leq a-1$ are, up to multiplication by totally positive units, all the indecomposable elements in $\Z[\rho]$. 
\end{proposition}

For example, for $a=3$, we have $\O_K=\Z[\rho]$. Moreover, it can be computed by a program (we used Mathematica), that in fact $\min_{\delta\in\O_K^{\vee,+}}\textup{Tr}(\alpha\delta)= 2$ for every non-unit indecomposable $\alpha$ in $\Z[\rho]$ (for $a=3$). 
In fact, it seems likely that all non-unit indecomposables have the minimal trace 2 for all $a$, but we have not proved this yet.

\begin{cor} \label{cor:8.2}
	There is a diagonal universal form of rank $12a$ over $\Z[\rho]$.
Every diagonal universal form over $\Z[\rho]$ has rank at least $\frac{a-1}6$.
\end{cor}

\begin{proof}
	The existence of the universal form follows from Proposition \ref{pr:univ construct}: There are $a$ indecomposables up to multiplication by totally positive units, but squares have index 2 in the group of totally positive units, and so there are $2a$ indecomposables up to multiplication by the squares of units. We can take $6$ as an upper bound for the Pythagoras number.
	
	The lower bound follows from Proposition \ref{pr:univ upper diag}, as there are $m=a-1$ indecomposables of trace~2.
\end{proof}

\subsection{Example with trace $3$} \label{subsec:trace3}
Let us now discuss a family of orders which further differs from the simplest and Ennola's cubic fields. Let us consider a root $\rho$ of the polynomial $f(x)=x^3-(2a+2)x^2+a(a+2)x-1$ where $a\geq 2$. Note that this example comes from  \cite[Theorem (3.9)]{Th} by setting $r=a$ and $s=a+2$ in the latter family. We can also benefit from the knowledge of the systems of the fundamental units of $\Z[\rho]$: Thomas proved that it is formed by $\rho$ and $\rho-a$. Note that the first one is totally positive, the second one is not. The discriminant of $f$ is
$
4a^4+20a^3+28a^2-24a-59,
$ 
which is an irreducible polynomial in $a$. For example, among $2\leq a\leq 1000$, there are $920$ such fields with $\O_K=\Z[\rho]$. 

By again considering the parallelepipeds spanned by totally positive units, we find all the indecomposables:

\begin{proposition}
Let $\rho$ be a root of the polynomial $x^3-(2a+2)x^2+a(a+2)x-1$ where $a\geq 2$. Then $1$, $1-a\rho+\rho^2$, and 
\begin{align*}
((a+2)v+1)\rho-v\rho^2 &\quad \text{where } 1\leq v \leq a-1,\\
-1+((a+2)w+1)\rho-w\rho^2 &\quad \text{where } a\leq w \leq 2a-1
\end{align*} 
 are, up to multiplication by totally positive units, all the indecomposables in $\Z[\rho]$. 
\end{proposition}

For the ranks of universal forms, this Proposition implies a similar Corollary as \ref{cor:8.2} (all the elements with $w$ on the second line give trace 2 after multiplication by suitable $\delta\in\Z[\rho]^{\vee,+}$ that is independent of $w$).

For $a=3$ we actually get $\O_K=\Z[\rho]$. For the indecomposable $11\rho-2\rho^2$ (i.e., $v=a-1=2$), we found an explicit element $\delta\in\Z[\rho]^{\vee,+}$ such that $\Tr(\delta(11\rho-2\rho^2))=3$. We further checked by a program in Mathematica that this trace is minimal for this specific indecomposable integer. So $\Q(\rho)$ is an example of a cubic field for which the statement of Corollary \ref{cor:mintracesimpl} is not true!

\medskip

Given a totally real number field $K$, denote $t(K)$ the maximum of these minimal traces of indecomposables in $\co_K^+$.
In Section \ref{sec:quadr} we established that $t(K)=1$ for every real quadratic field $K$. For cubic fields, in Section \ref{sec:indeco} we proved that $t(K)=2$ for every simplest cubic field with $\co_K=\Z[\rho]$; and now we found a field with $t(K)=3$. In particular, $t(K)$ does not depend only on the degree $d$ of $K$.

The behavior of $t(K)$ provides many exciting open problems: Can $t(K)$ be arbitrarily large? Is it bounded at least in a fixed degree? Is 3 the maximum of $t(K)$ over all cubic fields $K$?

\end{document}